\documentclass{amsart}
\usepackage[utf8]{inputenc}
\usepackage[T1]{fontenc}
\usepackage[backend=biber,
            style=ieee,
            doi=false,
            url=false,
            isbn=false,
            style=alphabetic
            ]{biblatex}
\usepackage{amsmath}
\usepackage{amsfonts}
\usepackage{amssymb}
\usepackage{thmtools}
\usepackage{setspace}
\usepackage{amsthm}
\usepackage{cancel}
\usepackage{todonotes}
\usepackage{graphicx}
\usepackage{hyperref}
\usepackage{lmodern}
\usepackage[capitalise]{cleveref}
\usepackage{enumitem}
\usepackage{tikz-cd}
\usepackage{multicol}
\usepackage{datetime}
\usepackage{xcolor}

\crefname{enumi}{}{}
\Crefname{enumi}{}{}
\usetikzlibrary{decorations.markings}
\tikzset{negated/.style={
        decoration={markings,
            mark= at position 0.5 with {
                \node[transform shape] (tempnode) {$\backslash$};
            }
        },
        postaction={decorate}
    }
}
\theoremstyle{plain}
\declaretheorem[name={Theorem}, numberwithin=section]{thm}
\declaretheorem[name={Lemma}, sibling=thm]{lemma}

\declaretheorem[name={Proposition}, sibling=thm]{proposition}
\declaretheorem[name={Claim}, numberwithin=thm]{claim}
\declaretheorem[name={Corollary}, sibling=thm]{corollary}

\theoremstyle{definition}
\declaretheorem[name={Definition},numberwithin=section]{definition}

\DeclareMathOperator{\Fin}{Fin}

\DeclareMathOperator{\Inf}{Inf}

\DeclareMathOperator{\biemb}{\approx}
\DeclareMathOperator{\enters}{\searrow}
\newcommand{\A}{\mathcal{A}}

\newcommand{\B}{\mathcal{B}}
\newcommand{\C}{\mathcal{C}}

\newcommand{\E}{\mathcal{E}}

\newcommand{\LR}{\Leftrightarrow}
\newcommand{\ra}{\rightarrow}
\newcommand{\Ra}{\Rightarrow}
\newcommand{\La}{\Leftarrow}
\newcommand{\tdeg}[1]{\mathbf{#1}}
\newcommand{\embeds}{\hookrightarrow}

\newcommand{\ol}[1]{\overline{#1}}
\newcommand{\jhalt}{W^{\emptyset'}}
\newcommand{\cotwo}{\ol{\emptyset''}}
\renewcommand{\phi}{\varphi}
\setlist*[enumerate,1]{
 label=(\arabic*),
}
\newlist{inlinelist}{enumerate*}{1}
\setlist*[inlinelist,1]{
label=(\arabic*),
}

\title{Degrees of bi-embeddable categoricity of equivalence structures}
\subjclass[2010]{03C57}
\author[Bazhenov]{Nikolay Bazhenov}
\address{
Sobolev Institute of Mathematics, Novosibirsk, Russia\\
Novosibirsk State University, Novosibirsk, Russia}
\email{bazhenov@math.nsc.ru}
\urladdr{bazhenov.droppages.com}

\author[Fokina]{Ekaterina Fokina}
\address{Institute of Discrete Mathematics and Geometry, Vienna University of Technology, Austria}
\email{ekaterina.fokina@tuwien.ac.at}
\urladdr{dmg.tuwien.ac.at/fokina}

\author[Rossegger]{Dino Rossegger}
\address{Institute of Discrete Mathematics and Geometry, Vienna University of Technology, Austria}
\email{dino.rossegger@tuwien.ac.at}
\urladdr{dmg.tuwien.ac.at/rossegger}

\author[San Mauro]{Luca San Mauro}
\address{Institute of Discrete Mathematics and Geometry, Vienna University of Technology, Austria}
\email{luca.san.mauro@tuwien.ac.at}

\thanks{The first author was supported by RFBR, according to the research project No. 16-31-60058 mol\_a\_dk. The second, third and fourth author were supported by the Austrian Science Fund FWF through project~P~27527.
}

\begin{document}
\begin{abstract}
We study the algorithmic complexity of embeddings between bi-embeddable equivalence structures. We define the notions of computable bi-embeddable categoricity, (relative) $\Delta^0_\alpha$ bi-embeddable categoricity, and degrees of bi-em\-bed\-dable categoricity. These notions mirror the classical notions used to study the complexity of isomorphisms between structures. We show that the notions of $\Delta^0_\alpha$ bi-embeddable categoricity and relative $\Delta^0_\alpha$ bi-embeddable categoricity coincide for equivalence structures for $\alpha=1,2,3$. We also prove that computable equivalence structures have degree of bi-embeddable categoricity $\tdeg{0},\tdeg{0}'$, or $\tdeg{0}''$. We obtain results on index sets of computable equivalence structure with respect to bi-embeddability.
\end{abstract}
\maketitle
\section{Introduction}
The systematic study of the complexity of isomorphisms between computable copies of structures was initiated in the 1950's by Fr\"ohlich and Shepherdson~\cite{frohlich_effective_1956} and independently by Maltsev~\cite{maltsev_recursive_1962}.
The notions of computable categoricity (in the Russian tradition also called autostability) and relative computable categoricity are probably the most prominent in this line of research. A computable structure $\A$ is \emph{computably categorical} if for every computable copy $\B$ there is a computable isomorphism from $\B$ to $\A$. A structure $\A$ is \emph{relatively computably categorical} if for every copy $\B$ of $\A$ there is a $\deg(\A\oplus \B)$-computable isomorphism from $\B$ to $\A$.
For a survey of this topic see Fokina, Harizanov, and Melnikov~\cite{fokina_computable_2014}.

In this article we study the algorithmic complexity of embeddings between bi-em\-bed\-dable equivalence structures. The relation of bi-embeddability has attracted a lot of attention of specialists in computable structure theory and descriptive set theory in the recent years. Two structures $\A$ and $\B$ are \emph{bi-embeddable} if there is an embedding of $\A$ in $\B$ and \emph{vice versa}. In the literature $\A$ and $\B$ are sometimes called equimorphic~\cite{montalban_up_2005,montalban_equimorphism_2007,greenberg_ranked_2008}. Montalb\'an~\cite{montalban_up_2005} showed that any hyperarithmetic linear ordering is bi-embeddable with a computable one and together with Greenberg~\cite{greenberg_ranked_2008} they showed the same for hyperarithmetic Boolean algebras, trees, compact metric spaces, and Abelian p-groups. Fokina, Rossegger, and San Mauro~\cite{fokina_bi-embeddability_2017} studied degree spectra with respect to the bi-embeddability relation and noticed that any countable equivalence structure is bi-embeddable with a computable one.
For this reason, the study of the algorithmic complexity of embeddings is particularly interesting for this class of structures.

In this paper we focus on computability-theoretic properties of equivalence structures. An equivalence structure $\A=(A,E)$ has a computable subset $A$ of the natural numbers as its universe and is equipped with  an equivalence relation $E$ on $A$.
We identify such a structure $\A$  with its atomic diagram, which is the set of atomic formulas and negations of atomic formulas with parameters from $A$ that are true of $\A$ under a fixed G\"odel numbering. Our computability theoretic notions are standard and as in Soare~\cite{soare_turing_2016}.

Calvert, Cenzer, Harizanov, and Morozov~\cite{calvert_effective_2006}
initiated the study of computable categoricity for equivalence structures. Given a structure $\A$ and a structure $\B$ bi-embeddable with $\A$, we say that $\B$ is a \emph{bi-embeddable copy} of $\A$. We study the complexity of embeddings through the following notions analogous to computable categoricity and relative computable categoricity.

\begin{definition}\label{def:bicat}\leavevmode
\begin{itemize}
  \item A computable structure $\A$ is \emph{$\tdeg{d}$-computably bi-embeddably categorical} if any computable bi-embeddable copy of $\A$ is bi-embeddable with $\A$ by $\tdeg{d}$-computable embeddings. Here $\tdeg{d}$ is a Turing degree. For the case $\tdeg{d}=\tdeg{0}^{(n)}$, we usually use the term ``\emph{$\Delta^0_{n+1}$ bi-embeddably categorical}'' instead.
  \item A countable (not necessarily computable) structure $\A$ is \emph{relatively $\Delta^0_n$ bi-embeddably categorical} if for any bi-embeddable copy $\B$, $\A$ and $\B$ are bi-embeddable by  $\Delta^{\A \oplus \B}_n$ embeddings. A computable structure is relatively computably bi-embeddably categorical if $n=1$.
  \end{itemize}
   The definitions can be extended to arbitrary hyperarithmetic $\alpha$ in the usual way. In this paper we only focus on the cases of $n=0,1,2$.
\end{definition}

Fokina, Kalimullin, and R. Miller~\cite{fokina_degrees_2010} introduced the notion of a (strong) degree of categoricity of a structure. Given a computable structure $\A$ that is $\tdeg{d}$-computably categorical in the classical sense, the \emph{degree of categoricity} of $\A$ is the least degree that computes an isomorphism between any two computable isomorphic copies of $\A$. It is the \emph{strong degree of categoricity} of $\A$ if there are two computable copies $\A_0$ and $\A_1$ such that for any $f:\A_0\cong \A_1$, $f\geq_T \tdeg{d}$. Note that $\A$ might not have a (strong) degree of categoricity \cite{miller_d-computable_2009,fokina_categoricity_2016}. A Turing degree $\tdeg{d}$ is a (strong) degree of categoricity if there exists a structure having $\tdeg{d}$ as its (strong) degree of categoricity. This notion has seen a lot of interest over the last years. Fokina, Kalimullin, and R.\ Miller~\cite{fokina_degrees_2010} showed that all strong degrees of categoricity are hyperarithmetical. Csima, Franklin, and Shore~\cite{csima_degrees_2013} extended this result by showing that all degrees of categoricity are hyperarithmetical. R.\ Miller~\cite{miller_d-computable_2009} exhibited a field that does not have a degree of categoricity, and Fokina, Frolov, and Kalimullin~\cite{fokina_categoricity_2016} gave an example of a rigid structure without degree of categoricity. Recently, Csima and Stephenson~\cite{csima_finite_2017}, and independently, Bazhenov, Kalimullin, and Yamaleev~\cite{bazhenov_degrees_2016,bazhenov_degrees_????} found examples of structures that have degree of categoricity but no strong degree of categoricity. The question whether there exists a degree of categoricity that is not strong is still open. We give an analogous definition for bi-embeddability.
\begin{definition}
  The \emph{degree of bi-embeddable categoricity} of a computable structure $\A$ is the least Turing degree $\tdeg{d}$ that, if it exists, computes embeddings between any computable bi-embeddable copies of $\A$.
  If, in addition, $\A$ has two computable bi-embeddable copies $\A_0, \A_1$ such that for all embeddings $\mu: \A_0\embeds \A_1$, $\nu: \A_1\embeds \A_0$, $\mu\oplus \nu \geq_T \tdeg{d}$, then $\tdeg{d}$ is the \emph{strong degree of bi-embeddable categoricity} of $\A$.
\end{definition}
Csima and Ng~\cite{unpublished} showed that a computable equivalence structure has strong degree of categoricity $\tdeg{0}$, $\tdeg{0}'$, or $\tdeg{0}''$.
Our main result reflects theirs in our setting.
\begin{thm}\label{thm:dgcat}
  Let $\A$ be a computable equivalence structure.
  \begin{enumerate}
    \item\label{item:thmdgcat_1} If $\A$ has bounded character and finitely many infinite equivalence classes, then its degree of bi-embeddable categoricity is $\tdeg{0}$.
    \item\label{item:thmdgcat_2} If $\A$ has unbounded character and finitely many infinite equivalence classes, then its degree of bi-embeddable categoricity is $\tdeg{0}'$.
    \item\label{item:thmdgcat_3} If $\A$ has infinitely many infinite equivalence classes, then its degree of bi-embeddable categoricity is $\tdeg{0}''$.
  \end{enumerate}
  Thus, the degree of bi-embeddable categoricity of equivalence structures is either $\tdeg{0}$, $\tdeg{0}'$, or~$\tdeg{0}''$. Furthermore, the degrees of bi-embeddable categoricity of equivalence structures are strong.
\end{thm}
The proof of \cref{thm:dgcat} combines various theorems proved in \cref{sec:cateq,sec:relcat}. In these sections we also obtain results on the relation between classical notions of categoricity and bi-embeddable categoricity, summarized in \cref{fig:relations}.
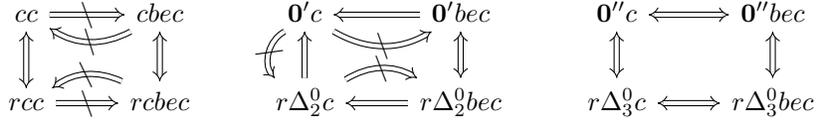
\begin{figure}[h]
\center
\begin{tikzcd}[arrows=Rightarrow]
  cc \arrow[negated]{r}\arrow[Leftrightarrow]{d} & cbec\arrow[negated, bend left=30]{l}\arrow[Leftrightarrow]{d} &  \tdeg{0}'c\arrow[ negated, bend right=30]{r}\arrow[negated, bend right=50]{d} &\arrow{l}\arrow[Leftrightarrow]{d} \tdeg{0}'bec & \tdeg{0}''c \arrow[Leftrightarrow]{r}\arrow[Leftrightarrow]{d} & \tdeg{0}''bec\arrow[Leftrightarrow]{d} \\
  rcc\arrow[negated]{r} & \arrow[negated, bend right=30]{l} rcbec &  r\Delta^0_2c\arrow{u} \arrow[ negated, bend left=30]{r}& \arrow{l}r\Delta^0_2bec & r\Delta^0_3c\arrow[Leftrightarrow]{r} & r\Delta^0_3bec

\end{tikzcd}
\caption{\label{fig:relations}Relations between categoricity for computable equivalence structures}
\end{figure}

In \cref{sec:cateq} we characterize the computably bi-embeddably categorical equivalence structures. In \cref{sec:relcat} we study $\Delta^0_2$ and $\Delta^0_3$ bi-embeddably categorical and relatively $\Delta^0_2$ and $\Delta^0_3$ bi-embeddably categorical equivalence structure. We show that all equivalence structures are relatively $\Delta^0_3$ categorical. We prove \cref{item:thmdgcat_2} and \cref{item:thmdgcat_3} of \cref{thm:dgcat} and study the relations between those notions summarized in \cref{fig:relations}.

In \cref{sec:index} we obtain results on the complexity of the index sets of equivalence structures with degrees of bi-embeddable categoricity $\tdeg{0},\tdeg{0}'$, and $\tdeg{0}''$.
 \section{Computable bi-embeddable categoricity}\label{sec:cateq}Given an equivalence structure $\A$ and $a\in A$ we write $[a]^\A$ for the equivalence class of~$a$; if it is clear from the context which structure is meant, we omit the superscript. The following notions are central to our analysis.
\begin{definition}
  Let $\A$ be an equivalence structure. A set $T\subseteq A$ is a \emph{transversal} of $\A$ if
  \begin{enumerate}
    \item for $x,y\in T$, if $x\neq y$, then $x\not \in [y]^\A$,
    \item and $A=\bigcup_{x\in T} [x]^\A$.
  \end{enumerate}
\end{definition}
\begin{proposition}
  Let $\A$ be an equivalence structure, then there is a transversal $T$ of $\A$ such that $T\leq_T \A$.
\end{proposition}
\begin{proof}
 For each equivalence class, we choose the least element in the class. We can do this computably in (the atomic diagram of) $\A$.
\end{proof}
\begin{definition}[Calvert, Cenzer, Harizanov, and Morozov~\cite{calvert_effective_2006}]
  Let $\A$ be an equivalence structure.
  \begin{enumerate}
    \item We say that $\A$ \emph{has bounded character}, or simply  \emph{is bounded}, if there is some finite $k$ such that all finite equivalence classes of $\A$ have size at most $k$. If $\A$ has bound $k$ on the sizes of its finite equivalence classes, we say that $\A$ is \emph{$k$-bounded}.
    \item \hfil $Inf^\A=\{a\in A: [a]^\A \text{ is infinite}\}\qquad Fin^\A=\{a\in A:[a]^\A \text{ is finite}\}$\hfill
  \end{enumerate}
\end{definition}
We will use the following
relativization of~\cite[Lemma 2.2]{calvert_effective_2006}.
\begin{lemma}\label{lem:compprop}
  Let $\A$ be an equivalence structure, then
  \begin{enumerate}
    \item\label{it:compprop1} For $k\in \omega$, $|[a]^\A|\leq k$ is $\Pi_1^\A$, $|[a]^\A| \geq k$ is $\Sigma_1^\A$, $|[a]^\A|=k$ is $\Delta^\A_2$,
    \item\label{it:compprop2} $Inf^\A$ is $\Pi^\A_2$, and $Fin^\A$ is $\Sigma^\A_2$,
  \end{enumerate}
\end{lemma}
\begin{proof}
  Ad \cref{it:compprop1}. A $\Pi_1^\A$ definition for $|[a]^\A|\leq k$ is
\[|[a]^\A|\leq k \LR \forall x_1,\dots, x_{k+1} \ \bigwedge_{1\leq i \leq k+1} x_i E a \ra \bigvee_{1\leq i<j\leq k+1} x_i=x_j,\]
  a $\Sigma_1^\A$ definition for $|[a]^\A|\geq k$ is
\[|[a]^\A|\geq k \LR \exists x_1,\dots, x_{k} \ \bigwedge_{1\leq i\leq k} x_i E a \land \bigwedge_{1\leq i<j\leq k} x_i\neq x_j,\]
  and a $\Delta^\A_2$ definition for $|[a]^\A|=k$ is then just the conjunction of $|[a]^\A|\leq k$ and $|[a]^\A|\geq k$.

  Ad \cref{it:compprop2}. The property $a\in Inf^\A$ has a $\Pi^\A_2$ definition by $\forall k\ |[a]^\A|\geq k$. It follows immediately that $a\in Fin^\A$ has a $\Sigma^\A_2$ definition.
\end{proof}
Our first goal is to characterize computably bi-embeddably categorical equivalence structures.
In~\cite{calvert_effective_2006} the following characterization of computably categorical equivalence structures was given.
\begin{thm}[Calvert, Cenzer, Harizanov, and Morozov]\label{thm:eqcatchar}
  Let $\A$ be a computable equivalence structure, then $\A$ is computably categorical if and only if
  \begin{enumerate}
    \item $\A$ has finitely many finite equivalence classes,
    \item\label{item:eqcatchar2} or $\A$ has finitely many infinite classes, bounded character, and at most one finite $k$ such that there are infinitely many classes of size $k$.
  \end{enumerate}
\end{thm}
\begin{thm}\label{thm:charbecat}
  An equivalence structure $\A$ is computably bi-embeddably categorical if and only if it has finitely many infinite equivalence classes and bounded character.
\end{thm}
\begin{proof}
  $(\La)$. Let $\A$ be $k$-bounded and let $l$ be the size of the largest equivalence class such that $\A$ has infinitely many equivalence classes of size $l$ (notice that $l$ might be $0$). Then the restriction $\A_{>l}$ of $\A$ to equivalence classes of size larger than $l$ is computably categorical, as the number of equivalence classes in $\A_{>l}$ is finite, i.e.,  the bi-embeddability type and the isomorphism type of $\A_{>l}$ coincide. Hence, if $\B$ is a bi-embeddable copy of $\A$, then $\B_{>l}$ is isomorphic to $\A_{>l}$. Non-uniformly fix a computable isomorphism $f: \A_{>l} \ra \B_{>l}$.

  Let $T_{\A_{>l}}$ be a transversal of $\A_{>l}$, and $T_{\B_{>l}}$ one of $\B_{>l}$. Clearly both $T_{\A_{>l}}$ and $T_{\B_{>l}}$ are finite and hence computable. Furthermore the equivalence classes of size $l$ have a c.e.\ transversal as
\[|[a]^\A| \leq l \LR \forall x\in T_{\A_{>l}}\ a\not\in [x]^\A\]
  and $|[a]^\A|\geq l$ is $\Sigma_1$. Let $(b_i)_{i\in \omega}$ be an enumeration of the transversal of the equivalence classes of size $l$ in $\B$ and let $(a_i)_{i\in \omega}$ be a computable  enumeration of $A$. We can define a computable embedding $\nu:\A \embeds \B$ by recursion as follows.
\[\nu(a_i) =\begin{cases}
      f(a_i) & \exists y\in T_{\A_{>l}}\ a_i\in [y]^\A\\
      b_k, \ k=\mu l [ \forall j<i\ \nu(a_j)\not \in [b_l]^\B] & \forall j<i\ a_i\not \in [a_j]^\A \\
      \mu x\in B [x\in [\nu(a_j)]^\B \land \forall l<i\  x\neq \nu(a_l) ] & \exists j<i\ a_i\in [a_j]^\A
    \end{cases}
\]
  The embedding of $\B$ in $\A$ is defined similarly.

  $(\Ra)$. We show that computable equivalence structures with unbounded character and without infinite equivalence classes are not computably bi-embeddably categorical. The proof for equivalence structures with finitely many infinite equivalence classes is analogous. By \cref{cor:delta02cat} below, equivalence structures with infinitely many infinite equivalence classes are not even $\Delta^0_2$ bi-embeddably categorical.

  Note that any two equivalence structures with unbounded character and the same number of infinite equivalence classes are bi-embeddable and that any embedding needs to map elements to elements in equivalence classes of at least the same size.
  Consider the equivalence structure $\A$ with universe $\bigcup_{i\in \omega} \{ \langle i,n\rangle : n\leq i \}$ where all elements with the same left column are in the same equivalence class. This structure is clearly a computable equivalence structure with computable size function $|\cdot|$. We build a computable equivalence structure $\B=(\omega,E^\B)$ in stages such that no partial computable function is an embedding of $\B$ in $\A$. We want to satisfy the following requirements.
\[P_e:\quad \phi_e \text{ is not an embedding of }\B\text{ in }\A\]
  We say that a requirement $P_e$ needs attention at stage $s$ if the restriction of the approximation $\phi_{e,s}$ of $\phi_e$ to elements in the structure $\B_s$ is a partial embedding of $\B$ in $\A$ not equal to $\emptyset$. The structure $\B$ is the limit of the structures $\B_s$ constructed as follows.

  \noindent\emph{Construction:}\\
  \emph{Stage $s=0$:} $\B_0$ is the singleton $\langle 0,0\rangle$.

  \noindent\emph{Stage $s+1$:}
    Check if there is a requirement $P_e$, $e\leq s$ that needs attention. If such $P_e$ exists, do the following.
    Choose the least requirement $P_e$ that needs attention. Then $\phi_{e,s}(\langle s,0\rangle) \downarrow=\langle i,n\rangle$ for some $\langle i, n\rangle$. Let $\B_{s+1}=\B_s \cup \{ \langle s,s+j\rangle: 0\leq j\leq s\}\cup \{\langle s+1,0\rangle\}$; put all elements with $s$ in the left column in the equivalence class of $\langle s,0\rangle$, and let $\langle s+1,0\rangle$ be a singleton.

    If no $P_e$ needs attention set $\B_{s+1}=\B_s \cup \{\langle s+1,0\rangle\}$ and let $\langle s+1,0\rangle$ be a singleton.

  \medskip
  \noindent\emph{Verification:}\\
  Assume towards a contradiction that $\phi_e$ is an embedding of $\B$ in $\A$ and no $\phi_j$, $j<e$ is an embedding. Then there is a stage $s$ such that $\phi_{e,s}$ is a partial embedding of $\B_{s}$ in $\A$ and no $P_j$, $j<e$, needs attention. Thus, $P_e$ receives attention at stage $s+1$ and $\phi_{e,s}(\langle s,0\rangle)\downarrow$; say $\phi_{e,s}(\langle s,0\rangle)=\langle i,n\rangle$. Then, by construction, the equivalence class of $\langle s,0\rangle$ in $\B_{s+1}$ is bigger than the one of $\langle i,s\rangle$. Thus, $\phi_e$ can not be an embedding. However, by construction of $\B$, every equivalence class is grown only once and has the size of an equivalence in $\A$ plus one. Thus, $\B$ is unbounded without infinite equivalence classes and hence, is bi-embeddable with $\A$.
\end{proof}

The following corollaries follow directly from \cref{thm:eqcatchar} and \cref{thm:charbecat}.
\begin{corollary}
  There is a computably bi-embeddably categorical equivalence structure that is not computably categorical.
\end{corollary}
\begin{corollary}
  There is a computably categorical equivalence structure that is not computably bi-embeddably categorical.
\end{corollary}

Calvert, Cenzer, Harizanov, and Morozov~\cite{calvert_effective_2006} showed that a computable equivalence structure is computably categorical if and only if it is relatively computably categorical. The analogous result holds in the context of bi-embeddability.
\begin{proposition}\label{prop:compcatiffrelcat}
  Let $\A$ be a computable equivalence structure. Then $\A$ is computably bi-em\-bed\-dably categorical if and only if it is relatively computably bi-em\-bed\-dably categorical.
\end{proposition}
\begin{proof}
  Relativization of the proof of \cref{thm:charbecat} ensures the result.
\end{proof}
 \section{$\Delta^0_2$ and $\Delta^0_3$ bi-em\-bed\-dable categoricity}\label{sec:relcat}In this section we characterize $\Delta^0_2$ and $\Delta^0_3$ bi-embeddably categorical equivalence structures. We will show that a computable equivalence structure is $\Delta^0_2$ ($\Delta^0_3$) bi-embeddably categorical if and only if it relatively so. We will also see that all equivalence structures are relatively $\Delta^0_3$ bi-embeddably categorical. This, together with the fact that any countable equivalence structure is bi-embeddable with a computable one~\cite{fokina_bi-embeddability_2017} gives a complete structural characterization of bi-embeddable categoricity for equivalence structures. We also establish the remaining parts of \cref{thm:dgcat}. At first we characterize $\Delta^0_2$ bi-embeddably categorical equivalence structures. We start by exhibiting a class of equivalence structures that is relatively $\Delta^0_2$ bi-embeddably categorical.
\begin{thm}\label{thm:biembdelta02}
  If $\A$ has finitely many infinite equivalence classes, then $\A$ is relatively $\Delta^0_2$ bi-embeddably categorical.
\end{thm}
\begin{proof}
  By \cref{thm:charbecat,prop:compcatiffrelcat} equivalence structures with bounded character are relatively computably bi-embeddably categorical and thus relatively $\Delta^0_2$ bi-embeddably categorical. It remains to show that equivalence structures with unbounded character and finitely many infinite equivalence classes are relatively $\Delta^0_2$ bi-embeddably categorical.

  Let $\A$ have finitely many infinite equivalence classes and unbounded character, and let $\B$ be a bi-embeddable copy of $\A$. Note that $\B$ must have the same number of infinite equivalence classes as $\A$. Fix transversals $T^\A$ and $T^\B$ of $Inf^\A$ and $Inf^\B$, respectively. Let $f: T^\A \ra T^\B$ be a bijection. As $T^\A$ and $T^\B$ are finite sets, they are computable. We define a $\Delta^{\A\oplus\B}_2$ embedding $\nu:\A\embeds \B$ by recursion. Let $(a_i)_{i\in \omega}$ be a computable enumeration of $A$.
\[\nu(a_0)=\begin{cases}
    f(t) & \text{if } \exists t\in T^\A\ a_0\in [t]^\A\\
    \mu x\in B [ |[x]^\B|\geq |[a_0]^\A| \land \forall t\in T^\B\ x\not \in [t]^\B] & \text{otherwise}
\end{cases}\]
  Assume $\nu$ has been defined for $a_0,\dots a_s$. We define $\nu(a_{s+1})$; there are three cases.

  \noindent \emph{Case 1:} $a_{s+1}$ is equivalent to an element for which $\nu$ has already been defined, i.e., $\exists j\leq s\ a_{s+1}\in [a_j]^\A$. Then
\[\nu(a_{s+1})=\mu x\in B [x\in [\nu(a_j)]^\B \land \forall i\leq s \ x\neq  \nu(a_i)].\]
  \noindent \emph{Case 2:} $a_{s+1}$ is not equivalent to any element for which $\nu$ has been defined and its equivalence class is infinite, i.e., $\exists t\in T^\A\ a_{s+1}\in [t]^\A$, then
\[\nu(a_{s+1})=  \mu x\in B [x\in [f(t)]^\B \land \forall i\leq s \ x\neq \nu(a_i) ].\]
  \noindent \emph{Case 3:} $a_{s+1}$ is not equivalent to any element for which $\nu$ has been defined and its equivalence class is finite. Then
\[\nu(a_{s+1})= \mu x\in B [|[x]^\B| \geq |[a_{s+1}]^\A| \land \forall t\in T^\B\  x\not \in [t]^\B \land \forall i\leq s\  x\not \in  [\nu(a_i)]^\B].\]

As $\A$ and $\B$ are both unbounded, at any stage $s$ of the construction we can find an element in B with an equivalence class greater than or equal to the one of $a_s$ in $\A$. Therefore, $\nu$ is an embedding. As $T^\A, T^\B$, and $f$ are computable and comparing the size of two equivalence classes is $\Delta^{\A\oplus \B}_2$, $\nu$ is $\Delta^{\A\oplus \B}_2$.
\end{proof}

The following is the relativization of the classical computability theoretic concepts of immune and simple sets to $\tdeg{0}'$.
\begin{definition}
  An infinite set $A$ is \emph{$\tdeg{0}'$-immune} if it contains no infinite set which is computably enumerable in $\tdeg{0}'$. A $\Sigma^0_2$ set $A$ is \emph{$\tdeg{0}'$-simple} if it is the complement of a $\tdeg{0}'$-immune set.
\end{definition}

\begin{thm}\label{thm:jumpimmuneeq}
     There is a computable equivalence structure $\A$ with infinitely many infinite equivalence classes such that $Fin^\A$ is $\tdeg{0}'$-simple. Hence, $Inf^\A$ is $\tdeg{0}'$-immune.
\end{thm}
\begin{proof}
  We build $\A$ with universe $\omega$ such that $Fin^\A$ is $\tdeg{0}'$-simple, i.e., for every infinite $\Sigma_2$ set $S$ the intersection of $Fin^\A$ and $S$ is nonempty.
  It has to satisfy the following requirements.
\[P_e:\quad |\jhalt_e|=\infty\quad \Ra\quad \jhalt_e\cap Fin^\A\neq \emptyset\]
     and the overall requirement that any transversal $T_{\Inf^\A}$ of $\Inf^\A$ is infinite.
\[G: \quad T_{\Inf^\A} \text{ is infinite.}\]
  Our strategy to satisfy a requirement $P_e$ is to pick a witness $x_e$ for $\jhalt_e$ and prevent the equivalence class of $x_e$ from growing any further.

  We will construct $\A$ in stages. Elements and equivalence classes can be in one of three states. An element  is \emph{blocked by $P_e$} if it is equivalent to a witness picked by $P_e$. During the construction we also \emph{designate} unblocked elements for expansion, i.e., we allow the equivalence class of such elements to grow in a later stage. Elements which are neither designated nor blocked are \emph{fresh}, these elements have equivalence classes of size~$1$.
     The set $\jhalt_{e,s}$ is the $\Sigma_2$ approximation of $\jhalt_{e}$ at stage $s$ and the set $Fin^\A_s$ is the set of blocked elements at stage $s$; we will have that $Fin^\A=\lim_s Fin^\A_s$.
     A strategy $P_e$ \emph{needs attention} at stage $s+1$ if
     \begin{equation}\tag{$\star$}\label{eq:const}\jhalt_{e,s}\cap Fin^\A_s = \emptyset \qquad \& \qquad \exists x > e^3\  (x\in \jhalt_{e,s}).
     \end{equation}
     \emph{Construction:}\\
     \emph{Stage $s=0$:} Let $A=\omega$ and $E^\A=\{ (x,x):x\in\omega\}$. Define $Fin^\A_s=\emptyset$.\\
     \emph{Stage $s+1$:} Assume we have built $\A_s$.
     \begin{enumerate}
       \item Choose the least $e<s$ such that $P_e$ needs attention. Take the least $x>e^3$ satisfying the second part of the matrix in \cref{eq:const}. Check if $x\leq s$. If so, then take the element $y$, $e^3<y\leq s$ which has been in the approximation the longest without interruption and declare its equivalence class as blocked by $P_e$. If not, then declare the equivalence class of $x$ as blocked by $P_e$. If $P_e$ receives attention for the first time, designate the least fresh element.
       \item Add to all designated equivalence classes a fresh element bigger than $s$.
       \item Check if for any $e<s$ there is an element $x$ blocked by $P_e$ that is not blocked by any $P_j$, $j<s$, and $x\in \jhalt_{e,s-1}$ but $x\not \in \jhalt_{e,s}$. If so, declare the equivalence class of $x$ as designated.
   \end{enumerate}
   \emph{Verification}:
   It is clear from the construction that $\A=\lim_s \A_s$ is a computable equivalence structure. The following two claims establish that it has the desired properties.
   \begin{claim}
     Every requirement $P_e$ is eventually satisfied.
   \end{claim}
   \begin{proof}
     By construction, no requirement $P_e$ injures a requirement $P_j$, $j\neq e$, so it is sufficient to consider them in isolation. As $\jhalt_{e,s}$ is a $\Sigma_2$ approximation we have that $x\in \jhalt_{e}$ iff there is a stage $s_0$ such that for all $t>s_0$, $x\in \jhalt_{e,t}$. In particular, there is a stage after which an element $x_0\in \jhalt_{e}$, $x_0>e^3$ will be in the approximation longer than any element $y\not \in \jhalt_{e}$.
     This element will be chosen by our strategy the next time that $P_e$ receives attention (which it will if its current witness is not in $\jhalt_{e}$). Hence, $P_e$ is satisfied in the limit.
   \end{proof}
   \begin{claim}
     The requirement $G$ is satisfied.
   \end{claim}
   \begin{proof}
    Assume that $P_e$ is the maximum requirement that acted at some stage in the construction. At most $e$ equivalence classes are blocked at this stage and because at every stage every designated equivalence class grows by one element, at most $e^2$ out of $e^3$ elements are blocked. By the same reasoning at most $e^2$ elements are designated for expansion. Hence, at least $e^3-2e^2$ fresh elements are left to expand the equivalence classes of designated elements and so, for $e>2$ there are enough fresh elements smaller than $e^3$ left to expand the designated elements. As every requirement that receives attention for the first time designates one fresh element, in the limit there are infinitely many infinite equivalence classes.
   \end{proof}
   \end{proof}
It is immediate that two equivalence structures $\A$ and $\B$ with infinitely many infinite equivalence classes are bi-embeddable. To obtain an embedding of $\A$ in $\B$ just map all equivalence classes of $\A$ to infinite equivalence classes of $\B$.
\begin{proposition}\label{prop:infnotdelta02}
  If an equivalence structure $\A$ has infinitely many infinite equivalence classes, then it is not $\Delta^0_2$ bi-embeddably categorical.
\end{proposition}
\begin{proof}
  Let $\A$ be a computable equivalence structure with infinitely many infinite equivalence classes and no finite equivalence classes, and take $\B$ as in \cref{thm:jumpimmuneeq}. Then, by the above argument, they are bi-embeddable and every embedding of $\A$ in $\B$ has as range an infinite subset of $Inf^\B$.

  Now, assume that $\nu:\A\embeds \B$ is $\tdeg{0}'$-computable. Then, its range is a $\Sigma^0_2$ set and an infinite subset of $Inf^\B$. But $Inf^\B$ is $\tdeg{0}'$-immune, a contradiction.
  Hence, no embedding of $\A$ in $\B$ is $\tdeg{0}'$-computable and therefore $\A$ is not $\Delta^0_2$ bi-embeddably categorical.
\end{proof}
\begin{corollary}\label{cor:delta02cat}
   An equivalence structure $\A$ is $\Delta^0_2$ bi-embeddably categorical if and only if it has finitely many infinite equivalence classes.
\end{corollary}
Calvert, Cenzer, Harizanov, and Morozov~\cite{calvert_effective_2006} characterized relatively $\Delta^0_2$ categorical computable equivalence relations. Their result relativizes.
\begin{proposition}[{Relativization of~\cite[Corollary 4.8]{calvert_effective_2006}}]\label{prop:chardelta02}
  A countable equivalence structure $\A$ is relatively $\Delta^0_2$ categorical if and only if $\A$ has finitely many infinite equivalence classes or $\A$ has bounded character.
\end{proposition}
\begin{corollary}
  An equivalence structure with bounded character and infinitely many infinite equivalence classes is relatively $\Delta^0_2$ categorical but not relatively $\Delta^0_2$ bi-embeddably categorical.
\end{corollary}
\begin{corollary}\label{cor:rel2iff2cat}
  Let $\A$ be a computable equivalence structure. Then $\A$ is $\Delta^0_2$ bi-embeddably categorical if and only if it is relatively $\Delta^0_2$ bi-embeddably categorical.
\end{corollary}
\begin{proof}
  By \cref{thm:biembdelta02} computable equivalence structures with finitely many infinite equivalence classes are relatively $\Delta^0_2$ bi-embeddably categorical and therefore, also $\Delta^0_2$ bi-embeddably categorical. It follows from \cref{prop:infnotdelta02} that these are all equivalence structures which are $\Delta^0_2$ bi-embeddably categorical. As relatively $\Delta^0_2$ bi-embeddably categorical equivalence structures have the same characterization the result follows.
\end{proof}
The analogue of \cref{cor:rel2iff2cat} does not hold for isomorphisms.
Kach and Turetsky~\cite{kach_categoricity_2009} gave an example of a $\Delta^0_2$ categorical but not relatively $\Delta^0_2$ categorical equivalence structure. Downey, Melnikov, and Ng~\cite{downey_categoricity_2015} showed that an equivalence structure $\A$ is $\Delta^0_2$ categorical iff the structure containing only one equivalence class of $\A$ for each finite size and all its infinite equivalence classes is $\Delta^0_2$-computably categorical.

We now proceed with the study of possible degrees of categoricity for equivalence structures.
\begin{definition}
  A function $f$ is \emph{limitwise monotonic} if there is a computable approximation function $h_f(\cdot,\cdot)$ such that
  \begin{enumerate}
    \item $f(x)=\lim_s h_f(x,s)$
    \item for all $x,s$ $h_f(x,s)\leq h_f(x,s+1)$
  \end{enumerate}
\end{definition}
It is not hard to see that for any limitwise monotonic function $f$, $f\leq_T \tdeg{0}'$. For more on limitwise monotonic functions and their applications see Downey, Kach, and Turetsky~\cite{downey_limitwise_2011}.
\begin{thm}\label{thm:dgcatunbounded}
  The degree of bi-embeddable categoricity of computable equivalence structures with unbounded character and finitely many infinite equivalence classes is $\tdeg{0}'$.
\end{thm}
\begin{proof}
First notice that there are countably many bi-embeddability types of equivalence structures with unbounded character and finitely many infinite equivalence classes. Namely exactly one for each number of equivalence classes of infinite size. We prove the theorem for equivalence structures with no infinite equivalence classes. However, the argument can be easily modified to accomodate equivalence structures with finitely many infinite equivalence classes.

We define the following function.
\[f(x):= 1 + \sum_{0\leq i\leq x, \phi_i(x)\downarrow} \phi_i(x)\]
Clearly $f$ is limitwise monotonic and dominates every partial computable function. By the domination theorem (see~\cite[Theorem 4.5.4]{soare_turing_2016}) it holds that for any set $D$ such that $f\leq_T D$, $D\geq_T \tdeg{0}'$; hence, in particular, $f\equiv_T\tdeg{0}'$.

We build a computable equivalence structure $\A_f$ with universe $\omega$ and no infinite equivalence classes in stages such that
\[|[\langle x,0\rangle]^{\A_f}| = f(x).\]
Let $h_f$ be the computable approximation for $f$. At stage $0$ of the construction,  let the universe of the approximation $\A_f$ be $\omega$ and put $\langle 0,n\rangle$ in the equivalence class of $\langle 0,0\rangle$ for $n<h_f(0,0)$. At stage $s+1$ check if for any $\langle x, 0\rangle$ $x\leq s$, $|[\langle x,0\rangle]|<h_f(x,s+1)$. If so, add $\langle x,s+1\rangle$ to the equivalence class.

Now consider the equivalence structure $\A$ with universe $\bigcup_{i\in \omega} \{ \langle i,n\rangle : n\leq i \}$ and where all elements with the same left column are in the same equivalence class. This structure is clearly a computable equivalence structure bi-embeddable with $\A_f$ and computable size function $|\cdot|$. Any embedding $\nu:\A_f \ra \A$ must map $[\langle x,0\rangle]^{\A_f} \mapsto [\langle y,0 \rangle]^\A$ with $|[\langle y,0\rangle]^\A|\geq [\langle x,0\rangle]^{\A_f}=f(x)$.
Consider the function $g(x)=|[\nu(\langle x,0\rangle)]^\A|$; as $|\cdot|$ is computable, $g\equiv_T \nu$ and as $\forall x\ g(x)\geq f(x)$, $g\equiv_T \nu \geq_T \tdeg{0}'$ by the domination theorem. As by \cref{cor:delta02cat} every computable equivalence structure with finitely many infinite equivalence classes is $\Delta^0_2$ bi-embeddably categorical the theorem follows.
\end{proof}

\begin{thm}\label{thm:biembdelta03}
  Equivalence structures are relatively $\Delta^0_3$ bi-embeddably categorical.
\end{thm}

\begin{proof}
By \cref{thm:biembdelta02}, equivalence structures with finitely many infinite equivalence classes are relatively $\Delta^0_2$ bi-embeddably categorical. Thus, it suffices to show that equivalence structures with infinitely many infinite classes are relatively $\Delta^0_3$ bi-embeddably categorical. Let $\A$ and $\B$ be equivalence structures with infinitely many infinite classes.  Recall that any two such equivalence structures are bi-embeddable. There is an embedding of $\A$ in $\B$ that maps every equivalence class in $\A$ to an infinite equivalence class in $\B$. As $\Inf^\B$ is $\Pi^\B_2$ there is at least one such embedding which is $\Delta^{\A\oplus\B}_3$. By the same argument there is a $\Delta^{\A\oplus\B}_3$ embedding of $\B$ in $\A$.
\end{proof}

The analogous result about classical relative $\Delta^0_3$ categoricity of equivalence structures is also true~\cite{calvert_effective_2006}, as every equivalence structure has a $\Sigma^c_3$ Scott family.

We close by proving the remaining parts of \cref{thm:dgcat}.
\begin{thm}\label{thm:dgcatinfinite}
  The degree of bi-embeddable categoricity of computable equivalence structures with infinitely many infinite equivalence classes is $\tdeg{0}''$.
\end{thm}
\begin{proof}
  We first build a computable equivalence structure $\A$ with the property that any infinite partial transversal of $Inf^\A$ computes $\tdeg{0}''$. Let $(\sigma_i)_{i\in \omega}$ be a computable $1-1$ enumeration of $2^{<\omega}$ and associate to every $\sigma_i$ an infinite set of witnesses $\{ \langle i,x,y \rangle : x,y\in \omega\}$. Elements of the form $\langle i,x,0\rangle$ will serve as witnesses while all other elements will be used to grow the equivalence classes.
  We will build $\A$ using a $\Pi_2$ approximation to $\cotwo$. Let $\cotwo_s$ be the $\Pi_2$ approximation at stage $s$ of our construction.

  \noindent \emph{Construction: }\\
  \emph{Stage s=0: } Let $A=\omega$ and $E^\A = \{ (x,x) : x\in \omega\}$. Furthermore, for all strings $\sigma_i$ in our computable enumeration of $2^{<\omega}$ designate witnesses $\langle i,0,0\rangle$.

  \noindent \emph{stage s+1: } Assume we have built $\A_s$.
  \begin{enumerate}
    \item For all witnesses with left column $i<s+1$ check if for some $x$ with $\sigma_i(x)=0$, $x\in \cotwo_{s+1}$. If $\langle i, j,0\rangle$ is a witness for such $\sigma_i$, discard it (never touch its equivalence class again during the construction) and designate the witness $\langle i, s+1,0\rangle$.
    \item For any $\sigma_i$, $i<s+1$ grow the equivalence class of its designated witness $\langle i,j,0\rangle$ to match $\min \{ |\{ t : t\leq s, x\in \cotwo_t\}| : \sigma_i(x)=1\}$ using fresh elements $\langle i, j, r\rangle$ with $r>s$.
  \end{enumerate}

  \noindent\emph{Verification: } We have to show that any infinite partial transversal of $\Inf^\A$ computes $\tdeg{0}''$. The following claim establishes the crucial part.
  \begin{claim}
  $\sigma_i\prec \cotwo \LR \exists y \langle i,y,0\rangle \in Inf^\A$
  \end{claim}
  \begin{proof}
  $(\Ra)$. Assume $\sigma_i \prec \cotwo$. As we have a $\Pi_2$ approximation to $\cotwo$ there is a stage $s$ such that no $x< |\sigma_i|$, $x\not \in \cotwo$ enters $\cotwo_t$ at any stage $t>s$. Hence, by construction there is a $j<s$ such that $\langle i,j,0\rangle$ has infinite equivalence class.

  $(\La)$. Assume $|[\langle i,j,0\rangle]|=\omega$ and let $\tau\prec \cotwo$ with $|\tau|=|\sigma_i|$. Assume, for some $x$, $\tau(x)=0$ and $\sigma_i(x)=1$. Then there are only finitely many stages $t$ such that $x\in \cotwo_t$. Thus, by construction, no equivalence class of elements with $i$ in the left column can become infinite. Therefore, if $\tau(x)=0$, then $\sigma_i(x)=0$ as well.
  Now assume that $\tau(x)=1$ and $\sigma_i(x)=0$. Then there are infinitely many stages $s$ such that $x\in \cotwo_s$.
  Hence, by construction, there can not be finite $j$ such that $|[\langle i,j,0\rangle]|=\omega$ because for any $j$ if $\langle i,j,0\rangle$ is designated at step $s$ there is a step $t>s$ such that $x\in \cotwo_t$. Let $t_0$ be the least such step, then $\langle i,j+1,0\rangle$ will be designated at step $t_0$ and no new elements will be added to the equivalence class of $[\langle i,j,0\rangle]$ at any stage $t\geq t_0$. Thus $[\langle i,j,0\rangle]$ is finite, $\tau(x)=1\Ra \sigma_i(x)=1$ and hence $\tau=\sigma$.
  \end{proof}
  Consider an infinite partial transversal $T$ of $Inf^\A$. To check whether some fixed $x\in \tdeg{0}''$, consider an enumeration of $T$. As all elements in $T$ code an initial segment of $\cotwo$ in their left column there is a finite stage $t$ and $y,z$, such that $\langle i,y,z\rangle \in T_t$ and $|\sigma_i|\geq x$. Hence, $x\in \tdeg{0}'' \LR \sigma_i(x)=0$ and we can find $\sigma_i$ uniformly. Therefore, $T\geq_T \tdeg{0}''$.

  Now consider the structure $\B$ with universe $\omega$ and where $\forall x\forall n \ \langle x, n \rangle \in [\langle x,0 \rangle]^\B$. It is a computable equivalence structure consisting only of infinite equivalence classes and  it clearly embeds into $\A$. To compute $\tdeg{0}''$ from any embedding $\nu: \B \embeds \A$ look at the strings coded by the left column of the images of elements of the form $\langle x,0\rangle$. By the above argument, after enumerating a finite number of images of such elements we can decide whether $x\in \tdeg{0}''$, hence $\nu \geq_T \tdeg{0}''$. Since, by \cref{thm:biembdelta03}, any equivalence structure is relatively $\Delta^0_3$ bi-embeddably categorical the theorem follows.
\end{proof}

At last we put together the pieces that prove \cref{thm:dgcat}.
\begin{proof}[Proof of \cref{thm:dgcat}]
  \cref{item:thmdgcat_1} follows directly from \cref{thm:charbecat}. \cref{thm:dgcatunbounded} proves that $\tdeg{0}'$ is the degree of bi-embeddable categoricity of equivalence structures with unbounded character and finitely many infinite equivalence classes. The two equivalence structures $\A_f$ and $\A$ constructed in the proof witness that $\tdeg{0}'$ is a strong degree of bi-embeddable categoricity for equivalence structures with unbounded character and no infinite equivalence classes. Similar structures can be easily constructed for equivalence structures with any finite number of infinite classes. This proves \cref{item:thmdgcat_2}.
  \cref{thm:dgcatinfinite} shows that $\tdeg{0}''$ is the degree of bi-embeddable categoricity of equivalence structures with infinitely many infinite classes. To see that it is a strong degree of bi-embeddable categoricity, consider the structures $\A$ and $\B$ constructed in the proof. Any embedding $\nu: \B \embeds \A$ computes $\tdeg{0}''$, hence, for any $\mu:\A\embeds\B$, $\mu \oplus \nu \geq_T \tdeg{0}''$.
\end{proof}
 \section{Index sets}\label{sec:index}Let $(\C_e)_{e\in \omega}$ be an enumeration of the partial computable equivalence structures, i.e., given a computable function $\phi_e: \omega\times \omega \ra \{0,1\}$, $\C_e$ has universe $\omega$ and
\[ xE^{\C_e}y :\LR \phi_e(x,y)=1.\]
Note that it is $\Pi^0_1$ to check whether $\C_e$ is indeed an equivalence structure.

We say that a set is $D^0_n$ if it is the difference of two $\Sigma_n$ sets, or equivalently, the intersection of a $\Sigma^0_n$ and a $\Pi^0_n$ set.
We make the following assumptions without loss of generality.
\begin{enumerate}
  \item Let $P$ be a $\Pi^0_n$ set, $n>1$, and $S$ a $\Sigma^0_{n-1}$ set such that
  \[ p\in P \LR \exists s \langle s,p\rangle\in S,\]
  then the witness $s$ is unique.
  \item Let $S$ be a $\Sigma^0_n$ set, $n>1$, and $P$ a $\Pi^0_{n-1}$ set such that
  \[ s\in S \LR \forall p \langle p,s\rangle \in S,\]
  then $\langle p,s\rangle \not \in S \Ra \langle p+1,s\rangle \not \in S$, i.e., the left columns form initial segments of $\omega$.
\end{enumerate}
We start by recalling a simple observation.
\begin{lemma}\label{lem:betypecompcat}
  For computably bi-embeddably categorical equivalence structures $\A$ we have that $\B\biemb \A$ if and only if
  \begin{enumerate}
    \item $\B$ has the same number of infinite equivalence classes as $\A$,
    \item $\B$ has the same bound as $\A$,
    \item and if $\A$ has infinitely many equivalence classes of size $n$ and for all $k>n$ there are only finitely many equivalence classes of size $k$, then for $m\geq n$, $\B$ has the same number of equivalence classes of size $m$ as $\A$.
  \end{enumerate}
\end{lemma}

\begin{thm}\label{thm:dgcompindex}
   Let $\A$ be a computable, computably bi-embeddably categorical equivalence structure.
   \begin{enumerate}
     \item\label{it:degcompindex1}  If $\A$ is finite then the index set $\{ \C_e:  \C_e\biemb \A \}$ is $D^0_1$-complete.
     \item\label{it:degcompindex2}  If $\A$ has infinitely many equivalence classes of size $n$ for some $n<\omega$, and no infinite equivalence classes, then the index set $\{ \C_e:  \C_e\biemb \A \}$ is $\Pi^0_2$-complete.
     \item\label{it:degcompindex3}  If $\A$ has $r>0$ infinite equivalence classes, then the index set $\{\C_e:\C_e\biemb \A\}$ is $\Pi^0_2$-complete.
   \end{enumerate}
 \end{thm}
 \begin{proof}
    Ad \cref{it:degcompindex1}. Assume $\A$ is finite, say $|A|=m$. Let $\theta_\A$ be the formula obtained from the atomic diagram by replacing the constants from $A$ by variables. Then the index set is definable by the following $D^0_1$ formula.
\[\C_e\biemb \A  \LR \C_e \cong \A \LR \exists x_1,\dots x_m (\theta_\A(x_1,\dots, x_m)) \land \forall x_1,\dots x_{m+1} (\bigvee_{1\leq i<j\leq m+1} x_i=x_j).\]

    To see that the index set is $D^0_1$-hard, we define a computable function $g$ such that $\C_{g(e,i)}\biemb \A$ if and only if $e\in \emptyset'$ and $i\in \ol{\emptyset'}$. Fix an element $a$ of $\A$. Let $\B$ be a computable equivalence structure isomorphic to $\A\setminus [a]^\A$; say $|[a]^\A|=n$. We build a computable function $f$ and a structure $\E_{f(e,i)}$ disjoint from $\B$ in stages such that $\C_{g(e,i)}=\B \oplus \E_{f(e,i)}$.

    Let $\emptyset'_s, \E_{f(e,i),s}$ be the approximations to $\emptyset'$ and $\E_{f(e,i)}$, respectively, at stage $s$. Let $\E_{f(e,i),0}=\emptyset$ and assume we have defined $\E_{f(e,i),s}$. To define $\E_{f(e,i),s+1}$ check if (i) $e\enters \emptyset'_s$ and (ii) $i\enters \emptyset'_s$.
    The structure $\E_{f(e,i),s+1}$ extends $\E_{f(e,i),s}$ as follows. If (i), add a new equivalence class of size $n$ to $\E_{f(e,i),s+1}$  by using the elements $2(s+j)$ for $j\in 1,\dots, n$.
If (ii), add a new equivalence class of size $n+1$ by using the elements $2(s+j)+1$ for $j\in 1,\dots n+1$.

    Let $\E_{f(e,i)}=\lim_s\E_{f(e,i),s}$. It is now easy to see that $\C_{g(e,i)}=\B\oplus \E_{f(e,i)}$ is bi-embeddable with $\A$ (in this case even isomorphic) if $e\in \emptyset'$ and $i\in \ol{\emptyset'}$.

    Ad \cref{it:degcompindex2}. Assume $n$ is the maximal size such that $\A$ has infinitely many equivalence classes of size $n$. Let $\A_{>n}$ be the substructure of $\A$ restricted to classes bigger than $n$. Then $\A_{>n}$ is finite. The index set is definable by the following $\Pi^0_2$ formula.
\[\C_e \biemb \A \LR \forall x \exists y>x  ([y]^{\C_e}\geq n \land \forall z<y\  \neg zEy )\land \C_{e,>n}\cong\A_{>n}\]
    To see that it is hard, consider the $\Pi^0_2$ complete set $\Inf=\{ e: W_e \text{ is infinite}\}$. We will build a computable function $g$ such that $\C_{g(e)}\biemb \A \LR e\in \Inf$. Fix a computable equivalence structure $\B$ isomorphic to $\A_{>n}$. Our desired structure $\C_{g(e)}$ will be the disjoint union of $\B$ with the structure $\E_{f(e)}$, where $f$ is a computable function constructed as follows. Let $W_{e,s}$ be the computable approximation to $W_{e}$ after $s$ stages of our construction; we make the standard assumption that $x\in W_{e,s}\Ra x<s$. Assume we have defined $\E_{f(e),s}$ and are at stage $s+1$ of the construction. The structure $\E_{f(e),s+1}$ extends $\E_{f(e),s}$ by a new equivalence class of size $n$ for every $x\enters W_{e,s}$, i.e., if $x\enters W_{e,s}$ let $\langle x, s\rangle ,\dots ,\langle x,s+n\rangle\in \E_{f(e),s+1}$ and set them to be equivalent.
    This finishes the construction; let $\E_{f(e)}=\lim_s \E_{f(e),s}$.

    By construction $\E_{f(e)}$ has only equivalence classes of size $n$ and it has infinitely many of those if and only if $W_e$ is infinite. As $\B\cong \A_{>n}$, we have that \[\C_{g(e)}=\B \oplus \E_{f(e)} \biemb \A \LR e\in \Inf.\] Thus the index set is $\Pi^0_2$ complete.

    Ad \cref{it:degcompindex3}. To see that it is in $\Pi^0_2$ one has to consider two cases. Either the finite part of $\A$ is as in \cref{it:degcompindex1} or as in \cref{it:degcompindex2}. In any case, let $k$ be the bound. If we are in case \cref{it:degcompindex1}, let $m$ be the number of elements in the finite part of $\A$. If we are in case \cref{it:degcompindex2}, let $m$ be the number of elements in the finite part of $\A$ restricted to equivalence classes bigger than $n$, where $n$ is as above.

If we are in case \cref{it:degcompindex1} we can define the index set by
    \begin{multline}\tag{$\ast$}\label{eq:defformula} \C_e\biemb \A \LR
      \forall y_1,\dots,y_{m+1}(\bigwedge_{1\leq i\leq m+1} [y_i]\leq k \ra \bigvee_{1\leq i<j\leq m+1} y_i=y_j)\\
      \land\forall y_1, \dots, y_{r+1} (\bigwedge_{1\leq i\leq r+1} [y_i]>k \ra \bigvee_{1\leq i<j \leq r+1}  y_i E y_j)\land  \exists x_1,\dots, x_m (\theta_{\A_{fin}}(x_1,\dots, x_m))\\
      \land \forall y_1,\dots ,y_m ((\bigwedge_{1\leq i \leq m} [y_i]\leq k\wedge\bigwedge_{1\leq i<j\leq m} y_i\neq y_j) \ra \Theta_{\A_{fin}}(y_1,\dots, y_m))\\
      \land \exists x_1,\dots, x_r (\bigwedge_{1\leq i \leq r} [x_i]>k \land \bigwedge_{1\leq i<j\leq r} \neg x_i E x_j) \land \forall x ([x]>k \ra \exists y>x\ yEx)
      \end{multline}
    where $\theta_{\A_{fin}}$ is the formula obtained from the atomic diagram of the finite part of $\A$ by replacing the constants by variables and $\Theta_{\A_{fin}}$ is the disjunction over $\theta_{\A_{fin}}(x_1,\dots,x_m)$ permuting over all variables. Let the formula in \cref{eq:defformula} be $\phi_{\A}$. If we are in case \cref{it:degcompindex2} the defining formula is
      \[ \C_e \biemb \A \LR \forall x \exists y>x  ([y]^{\C_e}\geq n) \land \phi_{\A_{>n}}'\]
    where $\phi_{\A_{>n}}'$ is as above with the slight difference that the second and third universal quantifiers now range over all $x$ where $n<[x]\leq k$ instead of only $[x]\leq k$. The two formulas are easily seen to be $\Pi^0_2$.

    For the hardness consider the $\Pi^0_2$ complete set $\Inf=\{e: W_e\text{ is infinite}\}$ and fix a computable structure $\B$ without infinite equivalence classes isomorphic to the finite part of $\A$. We build a computable function $g$ such that
    \[\C_{g(e)}=\B\oplus \E_{f(e)}\biemb \A \LR e\in \Inf\]
    where $f$ is again a computable function. We proof the hardness for the case that $\A$ has one infinite equivalence class, the case for $r>1$ is similar. The construction of $\E_{f(e)}$ is in stages, at stage $0$ the universe of $\E_{f(e),0}$ is empty. Assume we have defined $\E_{f(e),s}$ and are at stage $s+1$ of the construction.

    For any $x<s$ such that $x\enters W_{e,s}$, add $\langle x,s+1\rangle$ to $\E_{f(e),s+1}$ and make it equivalent to all elements already in $\E_{f(e),s}$. It is easy to see that the structure $\E_{f(e)}=\lim_s \E_{f(e),s}$ has an infinite equivalence class if and only if $e\in \Inf$ and thus $\C_{g(e)}\biemb \A$ if and only if $e\in \Inf$.
 \end{proof}
\begin{thm}\label{thm:indexsetunbounded}
  Let $\A$ be an equivalence structure with degree of categoricity $\tdeg{0}'$. Then the following holds.
  \begin{enumerate}
    \item\label{it:thmindexunb1} If $\A$ has no infinite equivalence classes, then the index set $\{\C_e : \C_e \biemb \A\}$ is $\Pi^0_3$-complete.
    \item\label{it:thmindexunb2} If $\A$ has $0<k<\omega$ infinite equivalence classes, then the index set $\{\C_e : \C_e \biemb \A\}$ is $D^0_3$-complete.
    \end{enumerate}
\end{thm}
\begin{proof}
  Assume that $\A$ has $0<k<\omega$ infinite equivalence classes. Then the index set is definable by
\begin{multline*}\C_e\biemb \A \LR \forall x (x\in \Fin^{\C_e} \ra \exists y ( y\in \Fin^{\C_e} \land [y]\geq [x]))\\
  \land \exists x_1,\dots, x_k\in \Inf^{\C_e} (\bigwedge_{1 \leq i<j\leq k}\neg x_i E x_j) \\
  \land \forall x_0, x_1, \dots, x_k \in \Inf^{\C_e} ( \bigvee_{1\leq i <j\leq k} x_i E x_j).
\end{multline*}
The part of the formula defining the finite equivalence classes is $\Pi^0_3$ and the part defining the infinite equivalence classes is $\Sigma^0_3$. Thus, the formula is $D^0_3$. If $\A$ has no infinite equivalence classes then the part of the defining formula defining the infinite classes becomes $\forall x \ x\in Fin^{\C_e}$, a $\Pi^0_3$ formula. Hence, in this case the formula is $\Pi^0_3$.

To show the completeness of \cref{it:thmindexunb2} we will define a computable function $f$ such that for every $\Pi^0_3$ set $P$ and every $\Sigma^0_3$ set $S$
\[\C_{f(p,e)}\biemb \A \LR p\in P \land e\in S.\]
There is a $\Pi^0_2$ set $Q$ such that for all $e$
\[e\in S\LR \exists x \langle x,e\rangle \in Q.\]
Then, as $\Inf=\{e: W_e\text{ is infinite}\}$ is $\Pi^0_2$ complete we have that there is a computable set $T$ such that
\[
     \langle x,e\rangle \in Q \LR \{t:\langle x,e,t\rangle \in T\} \text{ is infinite.}
\]
In the same vein, as $\Fin=\overline{\Inf}$ is $\Sigma^0_2$ complete we have that there is a computable set $R$ such that
\[
      p \in P \LR \forall x\left( \{ r: \langle x,p,r\rangle \in R\} \text{ is finite}\right).
\]
We build the structure $\C_{f(p,e)}$ in stages.
\medskip

\noindent\emph{Construction:}

\noindent\emph{Stage $0$:} The universe of $\C_{f(p,e)}$ is $\omega$ and $\C_{f(p,e)}$ has exactly one equivalence class for each finite size, i.e., for each $x\in \omega$, we set $\langle 2x,0,0\rangle E \langle 2x,0, i\rangle$ for $i\in 0,\dots,x$. All other elements are singletons.

\noindent\emph{Stage $s+1=2j$:} For every $x<j$ check whether $\langle x,p,j\rangle\in R$. If so set $\langle 2x,0, 0\rangle E \langle 2x,0,j\rangle$.

\noindent\emph{Stage $s+1=2j+1$:} Check if for $x<j$, $\langle x,e,j\rangle\in T$ and if so, set $\langle 2x+1,i, 0\rangle E \langle 2x+1,i, j\rangle$ for $0\leq i<k$.
\medskip

\noindent\emph{Verification:}
\noindent $p\in P,e\in S$: The construction at even stages, taking care of the $\Pi^0_3$ part, will not contribute any infinite equivalence classes and $\C_{f(p,e)}$ is unbounded from the beginning. To see that it has $k$ infinite equivalence classes just notice that by our initial assumption at the beginning of this section there is exactly one $x$ such that $\{ t: \langle x,e,t\rangle \in T\}$ is infinite. Thus, the construction at odd stages guarantees that the equivalence classes of elements $\langle 2x+1, i ,0 \rangle $ with $0\leq i < k$ become infinite in the limit.

\noindent $p\not \in P,e\in S$: Then there is an $x$ such that $\{ t: \langle x,p,r\rangle \in R\}$ is infinite. By construction the equivalence class of $\langle 2x,0,0\rangle$ is infinite and our strategy for $S$ builds $k$ infinite equivalence classes. Thus, $\C_{f(p,e)}$ has $k+1$ infinite equivalence classes and hence, $\C_{f(p,e)}$ is not bi-embeddable with $\A$.

\noindent $p\in P, e\not \in S$: Then for no $x$ the set $\{ t: \langle x,e,t\rangle \in T\}$ is infinite. So, by construction, no equivalence class of elements with an odd number in the left column will grow to be infinite and all equivalence classes with an even number in the left column have finite equivalence classes as $\{ r: \langle x,p,r\rangle \in R\}$ is finite for all $x$.

\noindent $p\not \in P,s\not\in S$: Immediate from the above cases.

To prove completeness for \cref{it:thmindexunb1}, at every stage, we apply the strategy described above for even stages.
\end{proof}
\begin{thm}
  The index set $\{e: \C_e\text{ is computably bi-embeddably categorical}\}$ is $\Sigma^0_2$ complete.
\end{thm}
\begin{proof}
  Recall that an equivalence structure is computably bi-embeddably categorical if and only if it has bounded character and finitely many infinite equivalence classes. Thus the index set is definable by the following computable $\Sigma_2$ formula.
  \begin{multline*}\C_e \text{ is computably bi-embeddably categorical}\\ \LR \exists k \bigvee_{r\in \omega}( \forall x_1,\dots, x_{r+1} (\bigwedge_{1\leq i\leq r+1} |[x_i]|\geq k\ra \bigvee_{1\leq i<j \leq r+1} x_i E x_j)).\end{multline*}

  To see that it is hard consider the classical $\Sigma^0_2$ complete set $\Fin=\{ e: W_e \text{ is finite}\}.$ We build a computable function $f$ such that
\[\C_{f(e)} \text{ is computably bi-embeddably categorical} \LR W_e \text{ is finite}.\]

  Let $W_{e,s}$ be the computable approximation of $W_e$ at stage $s$. We construct $\C_{f(e)}$ in stages. At stage $0$ $\C_{f(e),0}$ has universe $\omega$ and the equivalence relation is the identity relation. Assume we have defined $\C_{f(e),s}$. To define $\C_{f(e),s+1}$ check for $x<s$ if $x\enters W_{e,s}$. If so declare $\langle x,s+i\rangle$ for $i\in 0,\dots, x$ to be equivalent. This finishes the construction. Let $\C_{f(e)}=\lim_s \C_{f(e),s}$. It follows directly from the construction that $\C_{f(e)}$ has bounded character if and only if $W_e$ is finite. This completes the proof.
\end{proof}
\begin{thm}\label{thm:indexset0'}
  The index set $\{ e: \C_e \text{ has degree of b.e.\ categoricity } \tdeg{0}'\}$ is $\Sigma^0_4$ complete.
\end{thm}
To proof that \cref{thm:indexset0'} is $\Sigma^0_4$ hard we use the function that we construct in the proof of \cref{thm:indexsetdeg02}. We thus first state and prove this theorem. Recall that any two equivalence structures with infinitely many infinite equivalence classes are bi-embeddable. Furthermore, by \cref{thm:dgcatinfinite}, equivalence structures with infinitely many infinite classes have degree of bi-embeddable categoricity $\tdeg{0''}$.
\begin{thm}\label{thm:indexsetdeg02}
  The index set $\{ e: \C_e \text{ has degree of b.e.\ categoricity~} \tdeg{0}''\}$ is $\Pi^0_4$-complete.
\end{thm}
\begin{proof}
  By \cref{thm:dgcatinfinite} the index set of the equivalence structures having degree of bi-embeddability categoricity $\tdeg{0}''$ is the same as the index set
\[\{ e: \C_e \text{ has infinitely many infinite equivalence classes}\}.\]
  This index set is clearly $\Pi^0_4$. To see that it is complete consider a $\Pi^0_4$ set $P$, then there is a computable set $S$ such that
\[e\in P \LR \forall x\exists y \forall u \exists v\ \langle x,y,u,v, e\rangle \in S.\]
  We now build a computable function $f$ such that
\[\C_{f(e)}\text{ has infinitely many infinite equivalence classes}.\LR e\in P\]
  The construction is in stages; the universe of $\C_{f(e),0}$ is $\omega$ and $E^{\C_{f(e),0}}=\{ (x,x): x\in \omega\}$.
  Assume we have defined $\C_{f(e),s}$ and are at stage $s+1$ of the construction. The structure $\C_{f(e),s+1}$ extends $\C_{f(e),s}$ as follows. For each $x,y,u,v\leq s+1$ such that $\langle x,y,u,v,e\rangle \in S$, set $\langle x,y,0,0\rangle E^{\C_{f(e),s+1}} \langle x,y,u,v\rangle$. Then proceed to the next stage.

  The desired structure $\C_{f(e)}$ is the structure in the limit of the construction. If $e\in P$, then $\C_{f(e)}$ has infinitely many infinite equivalence classes as for every $x$ there is a $y$ such that $\{ \langle u,v \rangle : \langle x,y,u,v,e\rangle\in S\}$ is infinite and by construction the elements having the same first and second column are in the same equivalence class. Assume $e\not \in P$, then there exists an $x_0$ such that for no $y_0$ the above set is infinite; so no equivalence class of elements with $x_0$ in the left column will be infinite. Then by our assumption the same holds for all $x>x_0$. Hence, $\C_{f(e)}$ has only finitely many infinite equivalence classes.
\end{proof}

\begin{proof}[Proof of \cref{thm:indexset0'}]
  The index set is definable by the $\Sigma^0_4$ formula
  \begin{multline*}
    \C_e\text{ has degree of b.e.\ categoricity }\tdeg{0}' \\ \LR \forall k (\exists x ( x\in Fin^{\C_e} \land [x]\geq k ))\land \exists r(T_{\Inf^{\C_e}}\leq r).
  \end{multline*}
  where $T_{\Inf^{\C_e}}$ is a $\Pi^0_2$ transversal of $\Inf^{\C_e}$.
  For the hardness we use the strategy used in the proof of \cref{thm:indexsetdeg02}. There, given a $\Pi^0_4$ set $P$, we define a computable function $f$ such that
\[\C_{f(e)}\text{ has degree of b.e.\ categoricity }\tdeg{0}'' \LR e\in P.\]
  Clearly, $\ol P$ is $\Sigma^0_4$ and using the same function we have that
\[\C_{f(e)}\text{ has degree of b.e.\ categoricity } \tdeg{0}\text{ or } \tdeg{0}' \LR e\not \in P\]
  since if $e\not\in P$, $f$ produces an equivalence structure with finitely many infinite equivalence classes. Notice that $\C_{f(e)}$ need not have unbounded character and thus might have degree of categoricity $\tdeg{0}$. Therefore, define a computable unbounded equivalence structure $\B$ and a function $g$ such that
\[\C_{g(e)}=\C_{f(e)}\oplus \B.\]
  The function $g$ is clearly computable and
\[\C_{g(e)}\text{ has degree of b.e.\ categoricity }\tdeg{0}' \LR e\in \ol P.\]
  This proves that the index set is $\Sigma^0_4$ hard.
\end{proof}
 \printbibliography
\end{document}